\documentclass[12pt]{article}

\usepackage{amsmath}
\usepackage{amsthm}
\usepackage{amssymb}
\usepackage{graphicx}
\usepackage{standalone}
\usepackage{comment}
\usepackage{color}
\usepackage{enumerate}

\usepackage{xcolor}

\usepackage[lite]{amsrefs}

\makeatletter
\usepackage{comment}
\let\wfs@comment@comment\comment
\let\comment\@undefined

\usepackage{changes}
\let\wfs@changes@comment\comment
\let\comment\@undefined

\newcommand\comment{%
    \ifthenelse{\equal{\@currenvir}{comment}}
    {\wfs@comment@comment}
    {\wfs@changes@comment}%
}

\definechangesauthor[name=Corrado, color=red]{COR}
\definechangesauthor[name=Valentino, color=blue]{VALE}\usepackage{lineno}
\definechangesauthor[name=Ferdinando, color=green]{FERD}\usepackage{lineno}
\usepackage{amssymb}

\usepackage[all,cmtip]{xy}

\usepackage{mathrsfs}

\usepackage{tabularx}
\usepackage{booktabs}
\usepackage[labelfont=bf,format=plain,justification=raggedright,singlelinecheck=false]{caption}

\theoremstyle{plain}
\newtheorem{thm}{Theorem}[section]

\newtheorem{lem}[thm]{Lemma}
\newtheorem{prop}[thm]{Proposition}

\newtheorem{cor}[thm]{Corollary}

\author{Valentino Smaldore, Corrado Zanella and Ferdinando Zullo}
\title{On the stabilizer of the graph of linear functions over finite fields}
\date{\today}

\newcommand{\F}{\mathbb F}
\newcommand{\Fq}{{{\mathbb F}_q}}
\newcommand{\fq}{{{\mathbb F}_q}}

\newcommand{\Fqn}{{{\mathbb F}_{q^n}}}
\newcommand{\fqn}{{{\mathbb F}_{q^n}}}
\newcommand{\Fqt}{{{\mathbb F}_{q^t}}}

\newcommand{\cD}{{\mathcal D}}

\newcommand{\cL}{{\mathcal L}}
\newcommand{\la}{\langle}
\newcommand{\ra}{\rangle}

\newcommand{\Rt}{\operatorname{R-}q^t\operatorname-}
\newcommand{\Lt}{\operatorname{L-}q^t\operatorname-}

\DeclareMathOperator{\PG}{PG}
\DeclareMathOperator{\N}{N}
\DeclareMathOperator{\Tr}{Tr}

\DeclareMathOperator{\GL}{GL}
\DeclareMathOperator{\im}{im}
\DeclareMathOperator{\C}{\mathcal{C}}

\theoremstyle{definition}%% Hans

\newtheorem{defn}[thm]{Definition}
\newtheorem{rem}[thm]{Remark}
\newtheorem{ex}[thm]{Example}

\begin{document}

\maketitle

%\comment[id=FERD]{Ecco come fare domande!}

\begin{abstract}
In this paper we will study the action of $\F_{q^n}^{2 \times 2}$ on the graph of an $\fq$-linear function of $\fqn$ into itself. In particular we will see that, under certain combinatorial assumptions, its stabilizer (together with the sum and product of matrices) is a field. We will also see some examples for which this does not happen. Moreover, we will establish a connection between such a stabilizer and the right idealizer of the rank-metric code defined by the linear function and give some structural results in the case in which the polynomials are partially scattered.
\end{abstract}

\section*{Introduction}

Let $f \in \mathbb{F}_q[x]$, the \textit{graph} of $f$ is defined as the following set of affine points
\[ \mathcal{G}_f=\{ (y,f(y)) \colon y \in \fq \} \subseteq \mathrm{AG}(2,q). \]
We can see the projective plane $\mathrm{PG}(2,q)$ as the union of $\mathrm{AG}(2,q)$ and the line at infinity $\ell_{\infty}$.
In coordinates, the points of $\mathrm{PG}(2,q)$ are of the form $\langle (x_0,x_1,x_2) \rangle_{\fq}$ for some $(x_0,x_1,x_2)\in \mathbb{F}_q^3\setminus \{(0,0,0)\}$ and we may assume that $\ell_{\infty}$ is the set of points defined by the vectors with last component equal to zero and so the points in $\mathrm{AG}(2,q)$ are those of the form $\langle (a,b,1) \rangle_{\fq}$ for some $a,b \in \fq$, which in $\mathrm{AG}(2,q)$ is defined by the pair $(a,b)$. 

The set of \textit{directions} of $f  \in \mathbb{F}_q[x]$ is defined as
\[\mathcal{D}_f=\{ PQ \cap \ell_{\infty} \colon P,Q \in \mathcal{G}_f, \,\,\, P \ne Q \},\]
where $PQ$ denotes the line through the points $P$ and $Q$.
Note that 
\[\mathcal{D}_f=\{\la (1,m,0)\ra_{\fq} \colon m \in D_f\},\]
where $D_f$ is the set of slopes of the lines used in $\mathcal{D}_f$, that is
\[D_f=\left\{ \frac{f(y)-f(z)}{y-z} \colon y,z \in \fq, \,\,\, y \ne z \right\}.\]
Combinatorial conditions on $\mathcal{G}_f$ and/or $\mathcal{D}_f$ can give algebraic properties on $f$; see for instance the well-celebrated results in \cite{Ball,Ball2} where some conditions on the intersections between $\mathcal{G}_f$ and the affine lines together with bounds on the number of directions yield some linearity conditions on $f$. 
In this paper we will mainly consider the case in which $f \in \fqn[x]$ and it is an $\fq$-linear function, i.e.\ $f $ is a linearized polynomial.
In this case,
\[\mathcal{D}_f=\{\la (y,f(y),0)\ra_{\fq} \colon y \in  \mathbb{F}_{q^n}^*\}\]
and
\[D_f=\left\{ \frac{f(y)}{y} \colon y \in \mathbb{F}_{q^n}^* \right\}.\]

%\comment[id=COR]{Ho fatto degli edit minori qui}

Since $f $ is $\fq$-linear,
the affine lines meet $\mathcal{G}_f$ in either zero points or in a power of $q$ points. 

In this paper we investigate the natural action of the group $(\mathbb{F}_{q^n}^{2\times 2},+)$ on $\mathcal{G}_f$ by considering the {set}
\[ \mathbb{S}_f=\{ A \in \mathbb{F}_{q^n}^{2\times 2} \colon A\mathcal{G}_f\subseteq\mathcal{G}_f  \}, \]
where $A \mathcal{G}_f=\left\{ A \left(  \begin{array}{cc} y \\ f(y) \end{array}\right) \colon y \in \fqn \right\}$.
{Clearly} $\mathbb{S}_f$, together with $+$ and $\cdot$ the usual sum and product of matrices in $\mathbb{F}_{q^n}^{2\times 2}$ and $\star$ the multiplication by a scalar in $\fq$, forms an $\fq$-algebra.

Section \ref{s:preliminaries} provides the necessary notions related to linearized polynomials, scattered and partially scattered polynomials, linear sets, their relations with the graph of a linear function, and rank-metric codes.

In Section~\ref{s:Sf} we will prove that when $f$ has \emph{low weight}, that is if for every affine line $\ell$
\begin{equation}\label{e:lw}
|\ell \cap \mathcal{G}_f|<q^{n/2},
\end{equation}
then $(\mathbb{S}_f,+,\cdot)$ is a field.
Counterexamples in the case that $f$ has not low weight are added.

In Section~\ref{s:PS} the set $\mathbb S_f$ will be determined for many
classes of linearized polynomials, containing all known scattered polynomials.
Proposition~\ref{p:weightps} shows that the partially scattered polynomials are almost of low weight. 
The $\Lt$partially scattered polynomials $f$ such that $\mathbb S_f$ is not a field are then completely characterized. 
As regards $\Rt$partially scattered polynomials, a class of polynomials is exhibited such that $\mathbb S_f$ is not a field.

In Section~\ref{s:RCf}, it is proved that the right idealizer of the rank distance code $\mathcal C_f$ associated with a linearized polynomial $f$
is isomorphic to $\mathbb S_f$.
This implies that many of the results in this paper can be translated
in terms of rank distance codes.

\section{Preliminaries}\label{s:preliminaries}

\subsection{Linearized polynomials, scattered and partially scattered polynomials}

Let $q$ be a prime power, and $n$  a positive integer.
A \emph{$q$-polynomial} (or \emph{linearized polynomial}) over the finite field $\fqn$ has form
$f =\sum_{i=0}^{k} a_i x^{q^i}\in\fqn[x]$;
if $a_k\neq0$ then the \emph{$q$-degree} of $f $ is $k$.
The set of linearized polynomials over $\fqn$ will be denoted as $L_{n,q}$.
Such set, equipped with the operations of sum and multiplication by elements of $\fq$ and the composition, results to be an $\fq$-algebra.
The quotient algebra $\mathcal{L}_{n,q}=L_{n,q}/(x^{q^n}-x)$ has the property that its elements are in one-to-one correspondence with the  $\fq$-linear
endomorphisms of $\fqn$.
Note that we can identify the elements of $\mathcal{L}_{n,q}$ with the $q$-polynomials having $q$-degree smaller than $n$.

Following \cite{LZ2}, let $f $ be a $q$-linearized polynomial over $\fqn$, $t$ a divisor of $n$ such that $1<t<n$, so that $n=tt'$.
%, and $\ell\in\{0,\ldots,n-1\}$.
%\comment[id=COR]{macro: $\backslash$Rt $\backslash$Lt $\backslash$ker $\backslash$im}
We say that $f $ is \emph{$\Lt$partially scattered} if for any $y,z\in\mathbb F_{q^n}^*$,
\begin{equation}\label{eq:condL}
\frac{f(y)}{y}=\frac{f(z)}{z}\Longrightarrow \frac{y}{z}\in\mathbb{F}_{q^t},
\end{equation}
and that $f $ is \emph{$\Rt$partially scattered} if for any  
$y,z\in\mathbb F_{q^n}^*$,
\begin{equation}\label{eq:condR}
\frac{f(y)}{y}=\frac{f(z)}{z}\,\mbox{ and }\, \frac{y}{z}\in\mathbb{F}_{q^t}\Longrightarrow \frac{y}{z}\in\fq.
\end{equation}
A polynomial $f$ which is both $\Lt$partially scattered and $\Rt$partially scattered is called \emph{scattered} (see \cite{Sheekey}).

Let $f $ and $g $ be two linearized polynomials over $\fqn$ and consider the two related graphs $\mathcal{G}_f$ and $\mathcal{G}_g$ in AG$(2,q^n)$.
We say that $f $ and $g $ are \emph{equivalent} if there exists $\varphi \in \mathrm{\Gamma L}(2,q^n)$ such that $\varphi(\mathcal{G}_f)=\mathcal{G}_g$, that is, there exist $A\in \mathrm{GL}(2,q^n)$ and $\sigma \in \mathrm{Aut}(\fqn)$ with the property that for each $x \in \fqn$ there exists $y \in \fqn$ satisfying
\[ A \left( \begin{array}{cc} x^\sigma\\ f(x) ^\sigma \end{array} \right)=\left( \begin{array}{cc} y\\ g(y) \end{array} \right), \]
see \cite[Section 1]{BMZZ} and \cite[Section 1]{CsMPq5}.

This definition of equivalence preserves the property of being $\Rt$ and $\Lt$partially scattered:

\begin{prop}\cite[Proposition 7.1]{BZZ}
Let $f $ and $g $ be two equivalent $q$-polynomials in $\mathcal{L}_{n,q}$.
If $f $ is $\Rt$partially scattered (resp.\ $\Lt$partially scattered), 
then $g $ is $\Rt$partially scattered (resp.\ $\Lt$partially scattered).
\end{prop}

\begin{proof}
 Since $\mathcal{G}_f$ and $\mathcal{G}_g$ are $\mathrm{\Gamma L}(2,q^n)$-equivalent, $\mathcal{G}_f$ and $\mathcal{G}_g$ are also $\mathrm{\Gamma L}(2t',q^t)$-equivalent. If $f $ is $\Rt$partially scattered, by \cite[Theorem 2.2]{LZ2}, $\mathcal{G}_f$ is a scattered subspace of $\mathrm{PG}(2t'-1,q^t)$. Therefore $\mathcal{G}_g$ is a scattered subspace of $\mathrm{PG}(2t'-1,q^t)$, and hence $g $ is $\Rt$partially scattered by \cite[Theorem 2.2]{LZ2}. Otherwise, if $f $ is $\Lt$partially scattered, by \cite[Remark 2.5]{LZ2}, $\mathcal{G}_f$ is scattered with respect to a normal $(t'-1)$-spread of $\mathrm{PG}(2t'-1,q^t)$. Therefore $\mathcal{G}_g$ is scattered with respect to a normal $(t'-1)$-spread of $\mathrm{PG}(2t'-1,q^t)$, and hence $g $ is $\Lt$partially scattered by \cite[Remark 2.5]{LZ2}. 
\end{proof}

\subsection{Linear sets and their relations with the graph of a linear function}\label{sec:linsetgraph}

Let $V$ be an $r$-dimensional $\fqn$-vector space and let $\Lambda=\PG(V,\fqn)=\PG(r-1,q^n)$. 
Let $U$ be an $\fq$-subspace of $V$ such that $\dim_{\fq}(U)=k$, then the set
\[ L_U=\{\langle u \rangle_{\fqn} \colon u\in U\setminus\{0\}\}
\subseteq\PG(r-1,q^n)\]
is said to be an $\fq$-\emph{linear set} of rank $k$.
The \emph{weight} of a point $P=\langle v\rangle_{\fqn} \in \Lambda$ in $L_U$ is defined as $w_{L_U}(P)=\dim_{\fq}(U \cap \langle v\rangle_{\fqn})$.

Furthermore, $L_U$ is called \emph{scattered} if it has the maximum number $\frac{q^k-1}{q-1}$ of points or, equivalently, if all points of $L_U$ have weight one.
Blokhuis and Lavrauw in \cite{BlLav} proved the following result. 

\begin{thm}\label{th:boundscatt}
    The rank of a scattered $\fq$-linear set in $\PG(r-1,q^n)$ is 
    upper bounded by $\frac{rn}2$.
\end{thm}

When $\dim_{\fq}(U)=n$ and $r=2$, then there exists a $q$-polynomial $f  \in \fqn[x]$ such that $L_U$ is mapped by a collineation of $\mathrm{PGL}(2,q^n)$ into
\[ L_f=L_{\mathcal{G}_f}=\{\langle (y,f(y)) \rangle_{\fqn} \colon y \in \mathbb{F}_{q^n}^*\}. \]
%where
%\[ U_f=\{ (y,f(y)) \colon y \in \fqn\}. \]
Note that the polynomial $f  \in \mathcal{L}_{n,q}$ is scattered if and only if $L_f$ is.
We will now describe the connection between linear sets of the projective line and graphs of linear functions.
Indeed, if $f  \in \fqn[x]$ is an $\fq$-linear function, then $\mathcal{D}_f$ coincides with the linear set $L_f$ defined by $f$ adding a zero component to each of its points. 
Assume that $\{i_1,\ldots,i_t\}$ is the set of weights of the points with respect to $L_f$ and let $\ell$ be any line in $\mathrm{PG}(2,q^n)$.
If $\ell$ meets $\ell_\infty$ in a point of  $\mathcal{D}_f$ of weight $j\in \{i_1,i_2,\ldots,i_t\}$ then $\ell$ meets $\mathcal{G}_f$ in exactly $q^j$ points or in zero points. 
Moreover, there are exactly $q^{n-i_j}$ lines hitting $\mathcal{G}_f$ in $q^{i_j}$ points for every point of weight $i_j$.
If $\ell$ meets $\ell_\infty$ in a point outside $\mathcal{D}_f$ then $\ell$ meets $\mathcal{G}_f$ in exactly one point. For more details we refer to \cite{FSzT}.
As a consequence, a low weight polynomial $f$ is a polynomial for which the associated linear set $L_f$ has all points of weight less than $n/2$.
These can give interesting constructions of Hamming metric codes with few weights, see \cite{NPSZcodes}.

\subsection{Rank-metric codes}

Rank-metric codes have been originally introduced by Delsarte \cite{Delsarte} in 1978. 
They have been intensively investigated in recent years because of their applications; we refer to \cite{sheekey_newest_preprint} for a survey on this topic.
The set $\mathbb{F}_q^{m\times n}$ of $m \times n$ matrices over $\fq$ can be endowed with the \emph{rank-metric} defined by
\[d(A,B) = \mathrm{rk}\,(A-B).\]
A subset $\C \subseteq \mathbb{F}_q^{m\times n}$ equipped with the rank-metric is called a \emph{rank-metric code}.
The \emph{minimum distance} of $\C$ is defined as
\[d = \min\{ d(A,B) \colon A,B \in \C,\,\, A\neq B \}.\]
We will denote the parameters of a rank-metric code $\C\subseteq\mathbb{F}_q^{m\times n}$ with minimum distance $d$ by $(m,n,q;d)$.
Delsarte showed in \cite{Delsarte} that the parameters of these codes must 
fulfill a Singleton-like bound.

\begin{thm}\label{th:Singleton}\cite{Delsarte}
If $\C$ is a rank-metric code with parameters $(m,n,q;d)$, then
\[ |\C| \leq q^{\max\{m,n\}(\min\{m,n\}-d+1)}. \]
\end{thm}

When equality holds, we say that $\C$ is a \emph{maximum rank distance} (\emph{MRD} for short) code.

We will be mainly interested in $\fq$-\emph{linear} rank-metric codes, that is $\fq$-subspaces of $\mathbb F_q^{m\times n}$.
Two $\fq$-linear rank-metric codes $\C$ and $\C'$ in $\mathbb{F}_q^{m\times n}$ are \emph{equivalent} if and only if there exist $X \in \mathrm{GL}(m,q)$, $Y \in \mathrm{GL}(n,q)$, and a field automorphism $\sigma$ of $\F_q$ such that
\[\C'=\{XC^\sigma Y \colon C \in \C\}.\]
The \emph{left} and \emph{right idealizers} of a rank-metric code $\mathcal{C}\subseteq\F_{q}^{m\times n}$ are defined as
\[ L(\C)=\{ Y \in \F_q^{m \times m} \colon YC\in \C\hspace{0.1cm} \text{for all}\hspace{0.1cm} C \in \C\},\]
\[ R(\C)=\{ Z \in \F_q^{n \times n} \colon CZ\in \C\hspace{0.1cm} \text{for all}\hspace{0.1cm} C \in \C\},\]
respectively.
They are powerful tools to study the equivalence issue among rank-metric codes.
These notions have been introduced by  Liebhold and Nebe in \cite[Definition 3.1]{LN2016}; they are invariant under equivalences of rank-metric codes. Further invariants have been introduced in \cite{GiuZ,NPH2}.
In \cite{LTZ2}, idealizers have been studied in details (under the name of \emph{middle} and \emph{right nuclei}) and the following result has been proved.

\begin{thm}\label{th:propertiesideal}\cite{LTZ2}
Let $\mathcal{C}$ and $\mathcal{C}^\prime$ be $\fq$-linear rank-metric codes of $\mathbb{F}_q^{m\times n}$.
\begin{itemize}
    \item If $\mathcal{C}$ and $\mathcal{C}^\prime$ are equivalent, then their left and right idealizers are isomorphic as $\fq$-algebras (\cite[Proposition 4.1]{LTZ2}).
    \item Let $\mathcal{C}$ be an MRD code with minimum distance $d>1$.
    If $m \leq n$, then $L(\C)$ is a finite field with $|L(\C)|\leq q^m$.
If $m \geq n$, then $R(\C)$ is a finite field with $|R(\C)|\leq q^n$.
In particular, when $m=n$, $L(\C)$ and $R(\C)$ are both finite fields (\cite[Theorem 5.4 and Corollary 5.6]{LTZ2}).
\end{itemize}
\end{thm}

We may see the nonzero elements of an $\fq$-linear rank-metric code $\mathcal{C}$ with parameters $(m,n,q;d)$ as:
\begin{itemize}
    \item matrices of $\F_q^{m\times n}$ having rank at least $d$ and with at least one matrix of rank exactly $d$;
    \item $\fq$-linear maps $V\to W$ where $V=V(n,q)$ and $W=V(m,q)$, having usual map rank at least $d$ and with at least one map of rank exactly $d$;
    \item when $m=n$, elements of the $\fq$-algebra $\mathcal{L}_{n,q}$ of $q$-polynomials over $\fqn$ modulo $x^{q^n}-x$, having rank at least $d$ and with at least one polynomial of rank exactly $d$.
\end{itemize}

\section{On the stabilizer of low weight polynomials}\label{s:Sf}

%Let $f(x)$ be a $q$-polynomial, and let 
%$G_f=\GL(2,q^n)_{\{U_f\}}$ be the setwise stabilizer of $U_f$ in 
%$\GL(2,q^n)$.
%\textcolor{red}{Define $G_f^o=G_f\cup\{O\}$ where $O$ is the $2\times2$ zero matrix.}
In this section we will study the action of the group $(\mathbb{F}_{q^n}^{2\times 2},+)$ on the graph of a low weight linearized polynomial $f$, by showing that the stabilizer $\mathbb{S}_f$ of its graph is a field. The following lemma is clear.
\begin{lem}\cite[Lemma 7.2]{BZZ}
  Let $f $ and $g $ be two $q$-polynomials over $\fqn$. If $f $ and
  $g $ are equivalent then $\mathbb{S}_f$ and $\mathbb{S}_g$ are isomorphic.   
\end{lem}  
We now prove that $\mathbb{S}_f$ is stable under the sum and the product of matrices.

\begin{prop}\label{p:semis}
If $A,B\in \mathbb{S}_f$, then $A+B,\ AB \in \mathbb{S}_f$.
\end{prop}
\begin{proof}
    Since $A\mathcal{G}_f$ and $B\mathcal{G}_f$ are both contained in $\mathcal{G}_f$, for any $x\in\Fqn$ there are $y,z\in\Fqn$ such that
    \[
      A\begin{pmatrix}x\\ f(x)\end{pmatrix}=\begin{pmatrix}y\\ f(y)\end{pmatrix}\quad\text{and}\quad
      B\begin{pmatrix}x\\ f(x)\end{pmatrix}=\begin{pmatrix}z\\ f(z)\end{pmatrix},
    \]
    hence
    \[(A+B)\begin{pmatrix}x\\ f(x)\end{pmatrix}=\begin{pmatrix}y+z\\ f(y)+f(z)\end{pmatrix}=\begin{pmatrix}y+z\\ f(y+z)\end{pmatrix},\]
so that $A+B \in \mathbb{S}_f$.
The same argument can be performed for the product.
\end{proof}

We are now ready to prove the main result of this section.

\begin{thm}\label{thm:Gfofield}
Let $f $ be a $q$-polynomial in $\mathcal{L}_{n,q}$.
If $f$ is a low weight polynomial, then $(\mathbb{S}_f,+,\cdot)$ is a field.
\end{thm}
\begin{proof}
    By Proposition~\ref{p:semis}, $\mathbb{S}_f$ is closed under the sum and the product of matrices.
    Since $\mathbb{S}_f$ contains the opposite of any of
    its elements and the inverse of its invertible matrices, it
    only remains to prove that if $A\in \mathbb{S}_f$, $A\neq O$,
    then $A\in \mathrm{GL}(2,q^n)$. 
    So, it is enough to prove
    that for any rank-one $2\times 2$ matrix $M$ with elements in $\Fqn$,
    $M\mathcal{G}_f$ is not contained in $\mathcal{G}_f$.
    To this aim, recall that from Section \ref{sec:linsetgraph} that since $f$ is a low weight polynomial, 
    $w_{L_f}(\langle v \rangle_{\fqn})=\dim_{\fq}(\langle v \rangle_{\fqn} \cap \mathcal{G}_f)<n/2$ for any  $v\in\mathbb F_{q^n}^2$.
    Let $Z\neq O$ be a column such that $MZ=O$ and let $C$ be a nonzero 
    column of $M$.

    Define $\mu: \mathcal{G}_f\rightarrow \F_{q^n}^2$, $(y,f(y))\mapsto M \left(\begin{array}{cc}  y\\ f(y)\end{array}\right)$.
    Since $\ker\mu\subseteq\la Z\ra_{\Fqn}\cap \mathcal{G}_f$, the $\Fq$-dimension
    of $\ker \mu$ is less than $n/2$ since $f$ is a low weight polynomial; so,
    $\dim_{\Fq}(\im\mu)>n/2$.

    Assume now that $M\mathcal{G}_f\subseteq \mathcal{G}_f$.
    This would imply $\im\mu\subseteq\la C\ra_{\Fqn} \cap \mathcal{G}_f$,
    contradicting $w_{L_f}(\la C\ra_{\Fqn})<n/2$.
\end{proof}

Theorem \ref{thm:Gfofield} allows us to find a large class of polynomials
for which $\mathbb S_f$ is a field.

\begin{prop}\label{lowDegree}
    Let $f $ be a $q$-polynomial in $\mathcal{L}_{n,q}$.
If $f$ has $q$-degree $k$ with $1<k<n/2$ then it is a low weight polynomial.
In particular $\mathbb S_f$ is a field.
\end{prop}
\begin{proof}
    By Section \ref{sec:linsetgraph}, we need the prove that the points in $L_f$ have weight at most $n/2-1$.
    Observing that 
    \[ w_{L_f}(\langle (1,m) \rangle_{\mathbb{F}_{q^n}})= \dim_{\mathbb{F}_q}(\ker(f -mx)) \]
    and that the $q$-degree of $f -mx$ is at most $n/2-1$, the assertion follows.
\end{proof}

As shown in the next examples, Theorem \ref{thm:Gfofield} cannot be improved in general.

\newpage

\begin{ex}\label{ex:traceSf}
Let $n=tt'$, $t'\ge2$, and let $f =\Tr_{q^n/q^t}(x)=x+x^{q^t}+\ldots+x^{q^{n-t}}$ be the trace function over $\F_{q^t}$.
    One point of $\mathcal{D}_f$ has weight $n-t$, and the remaining ones have weight $t$.
    Therefore, $f$ is not a low weight polynomial.
    %Let $a\in\F_{q^n}^*$ such that $\Tr_{q^n/q^t}(a)=0$.
    The matrix $\begin{pmatrix}a&b\\ c&d\end{pmatrix}$
    is an element of $\mathbb{S}_f$ if and only if for any $x \in \fqn$ there exists $y \in \fqn$ such that
    \[
      \begin{pmatrix}a&b\\ c&d\end{pmatrix}\begin{pmatrix}x\\ f(x)\end{pmatrix}=\begin{pmatrix}y\\ f(y)\end{pmatrix},
    \]
    from which we obtain the following polynomial identity
    \begin{equation}\label{eq:Sftrace} cx+(d-f(b))\sum_{j=0}^{t'-1}x^{q^{tj}}-\sum_{j=0}^{t'-1}a^{q^{tj}}x^{q^{tj}}=0. 
    \end{equation}
    From \eqref{eq:Sftrace} in the case $t'=2$ we only get two equations
    \[\left\{ \begin{array}{ll} c+d-f(b)-a=0,\\ d-f(b)-a^{q^t}=0, \end{array}\right.\]
    which imply that 
    \[\mathbb{S}_f=\left\{ \begin{pmatrix}
        a & b\\ a-a^{q^t} & f(b)+a^{q^t}
    \end{pmatrix} \colon a,b \in \fqn \right\}.\]
    From \eqref{eq:Sftrace} in the case $t'>2$, taking into account the coefficients
    of $x^{q^t}$ and $x^{q^{2t}}$ we get the following equations
    \[\left\{ \begin{array}{ll} d-f(b)-a^{q^t}=0,\\ d-f(b)-a^{q^{2t}}=0, \end{array}\right.\]
    which imply that $a\in \Fqt$.
    By comparing the coefficients of $x$ and $x^{q^t}$,
        \[\left\{ \begin{array}{ll} c=0,\\ d-f(b)-a=0. \end{array}\right.\]
Hence
    \[\mathbb{S}_f=\left\{ \begin{pmatrix}
        a & b\\ 0 & f(b)+a
    \end{pmatrix} \colon a \in \Fqt, b \in \fqn \right\}.\]
    In both cases, $\mathbb{S}_f$ contains non invertible nonzero matrices and hence it is not a field.
    In particular for $t'=2$ this shows that the bound~\eqref{e:lw} in
    the definition of a low weight polynomial is sharp in a sense.
\end{ex}

The next example shows that for every divisor $t$ of $n$ we can find a linearized polynomial $f$ for which the stabilizer of its graph is a field isomorphic to $\F_{q^t}$.

\begin{ex}
Let $n=t't$ with $t'>2$, $t\geq2$, and let $f =x^q+\eta x^{q^{t+1}} \in \mathcal{L}_{n,q}$ for $\eta \in \mathbb{F}^*_{q^n}$.
Then $f$ is an R-$q^t$-partially scattered polynomial if and only if $\N_{q^n/q^t}(\eta)\ne -1$, because of \cite[Proposition 2.10]{LZ2}. 
For all $\eta$, $f $ has $q$-degree strictly less than $n/2$ and by Proposition \ref{lowDegree} $f $ is a low weight polynomial and so $\mathbb{S}_f$ is a field which we can compute explicitly.
Our aim is to find the  matrices $\begin{pmatrix} a&b\\c&d \end{pmatrix}$ in $\mathbb{F}_{q^n}^{2\times 2}$ for which for every $x \in \fqn$ there exists $y \in \fqn$ such that
\[ \begin{pmatrix} a&b\\c&d \end{pmatrix} \begin{pmatrix} x\\
x^q+\eta x^{q^{t+1}}
\end{pmatrix}=\begin{pmatrix} y\\
y^q+\eta y^{q^{t+1}}
\end{pmatrix}, \]
from which we get the following polynomial identity
\begin{align}
&cx+d(x^q+\eta x^{q^{t+1}})=a^qx^q+b^q(x^{q^2}+\eta^q x^{q^{t+2}})+ \notag\\
& \eta [a^{q^{t+1}}x^{q^{t+1}}+b^{q^{t+1}}(x^{q^{t+2}}+\eta^{q^{t+1}} x^{q^{2t+2}})],\label{e:ognit}
\end{align}
and so we obtain $c=b=0$ and $d=a^q=a^{q^{t+1}}$, i.e. 
\begin{equation}\label{e:aaq}
\mathbb{S}_f=\left\{ \begin{pmatrix} a & 0 \\ 0 & a^q\end{pmatrix} \colon a \in \Fqt \right\}.
\end{equation}
\end{ex}

When $t'=2$ in the above example we have a linearized polynomial which is not a low weight polynomial. 
However if $t>2$ and $\N_{q^{2t}/q^t}(\eta)\neq1$, 
then $\mathbb S_f$ is once again a field, for
from~\eqref{e:ognit} reduced $\mod{x^{q^{2t}}-x}$ it holds
\[
\begin{cases}
c&=0\\ d&=a^q\\ 0&=b^q+b^{q^{t+1}}\eta^{q^{t+1}+1}\\
d&=a^{q^{t+1}}\\ 0&=b^q\eta^q+b^{q^{t+1}}\eta.
\end{cases}
\]
By combining the third and fifth equation, $b^q(1-\eta^{q+q^{t+1}})=0$,
leading to $b=0$ and by the second and fourth equation we get~\eqref{e:aaq}. If $\eta^{q^{t}+1}=1$, then third and fifth equation are equivalent and, by $\eta=s^{q^{t}-1}$, $s\in\mathbb{F}_{q^{2t}}^*$, we get
$$0=b^{q}+b^{q^{t+1}}s^{q-q^{t+1}+q^{t}-1}=b^{q}+b^{q^{t+1}}s^{(1-q)(q^{t}-1)},$$
that is $b^q \in \ker(x^{q^t}+s^{(1-q)(q^t-1)}x)$ which has dimension $t$.
Therefore,
\begin{equation*}
\mathbb{S}_f=\left\{ \begin{pmatrix} a & b \\ 0 & a^q\end{pmatrix} \colon a \in \Fqt, b^q \in \ker(x^{q^t}+s^{(1-q)(q^t-1)}x) \right\}.
\end{equation*}Since $\mathbb{S}_f$ contains singular nonzero matrices it is not a field.

%And by $x=b^{q}$,
%$$x^{q^{t}}+s^{(1-q)(q^{t}-1)}x=0,$$
%$$\Big(\frac{x}{s^{q-1}}\Big)^{q^{t}-1}=-1.$$
%We have a set $S$ of $q^{t}$ such $b$ verifying the equation, and hence
%\begin{equation}
%\mathbb{S}_f=\left\{ \begin{pmatrix} a & b \\ 0 & a^q\end{pmatrix} \colon a \in \Fqt \right\}.
%\end{equation}

%By $a=0$, we get a non invertible element of $\mathbb{S}_f$, that therefore is not a field.

\section{Examples of low weight polynomials and the related field}\label{s:PS}

We will now see examples of low weight polynomials.

The first family of low weight polynomials is given by the scattered polynomials, as in this case the maximum size of intersection between the related graph and the affine lines is $q$. We list here the known examples of polynomials for which $\mathbb{S}_f$ has been already determined,
including also some non-scattered ones.

\begin{itemize}
\item $f =x^{q^s}\in \mathcal{L}_{n,q}$ with $\gcd(s,n)=1$, then $|\mathbb{S}_f|=q^n$, see \cite[Section 6]{CMPZ};
\item $f =\delta x^{q^s}+x^{q^{n(s-1)}}\in \mathcal{L}_{n,q}$ with $\gcd(s,n)=1$, $\delta\neq0$ and $n\geq 4$, then $|\mathbb{S}_f|=q^2$ if $n$ is even  and $|\mathbb{S}_f|=q$ if $n$ is odd, see \cite[Section 6]{CMPZ} (we call these polynomials \emph{LP polynomials} even in case they are not scattered);
%\comment[id=COR]{vero anche nel caso in cui non sono scattered, giusto?}
\item $f =\delta x^{q^s}+x^{q^{s+n/2}}\in \mathcal{L}_{n,q}$ with $\delta\neq0$, $n$ even and $\gcd(s,n)=1$, then  $|\mathbb{S}_f|=q^{n/2}$, see \cite[Corollary 5.2]{CMPZ};
\item $f =x^q+x^{q^3}+\delta x^{q^5}\in \mathcal{L}_{6,q}$ with $q$ odd and $\delta^2+\delta=1$, then $|\mathbb{S}_f|=q^{2}$, see \cite[Proposition 5.2]{CsMZ2018} and \cite[Section 4.4]{MMZ}.
\item $f =x^{q^{s}}+x^{q^{s(t-1)}}+\eta^{1+q^{s}}x^{q^{s(t+1)}}+\eta^{1-q^{s(2t-1)}}x^{q^{s(2t-1)}}\in \mathcal{L}_{n,q}$ with $q$ odd prime power, $t,s,n \in \mathbb{N}$ with $n=2t$, $t\geq5$, $\gcd(s,n)=1$ and $N_{q^{n}/q^{t}}(\eta)=-1$, then $|\mathbb{S}_f|=q^2$, see \cite[Proposition 3.4]{LZ3}.
\end{itemize}

\begin{rem}
 Note that in certain cases listed above it is not necessary to require the scattered conditions in order to get low weight polynomials. 
 See for instance the examples below.
\end{rem}

The next two examples are low weight polynomials with maximum size of intersection between the related graph and the affine lines equals $q^2$.

\begin{ex}
Let $f =x^q+\delta x^{q^{n-1}} \in \mathcal{L}_{n,q}$ 
with $n>4$ and $\N_{q^n/q}(\delta)=1$ (note that $f$ is not scattered, see \cite{Zanella}). Moreover, consider $(u,v)\in \fqn\times\fqn$ such that $u\ne 0$, then
\[ w_{L_f}(\langle (u,v) \rangle_{\fqn})=\dim_{\fq}(\ker(f -xu^{-1}v))=\ker(x^{q^2}+\delta^q x-x^qu^{-q}v^q)
\leq 2, \]
as the latter polynomial has degree $q^2$ and also note that for $u=0$ then $w_{L_f}(\langle (u,v) \rangle_{\fqn})=0$. Hence $f$ is a low weight polynomial.
Computing as in the previous examples the matrices in $\mathbb{S}_f$ we obtain that
\[\mathbb{S}_f=\left\{\left(\begin{array}{cccc} a & 0 \\
0 & a^q\end{array}\right) \colon a \in \fqn\cap \mathbb{F}_{q^2} \right\},\]
and so $|\mathbb{S}_f|$ is $q$ if $n$ is odd and $q^2$ if $n$ is even.
\end{ex}

\begin{ex}
Let $f =x^q+\delta x^{q^2} \in \mathcal{L}_{n,q}$ with $n>4$ and $\delta \ne 0$. This is not a scattered polynomial, see \cite{Zanella}. By Proposition \ref{lowDegree}, $f$ is a low weight polynomial.
Computing as in the previous examples the matrices in $\mathbb{S}_f$ we obtain that
\[\mathbb{S}_f=\left\{\left(\begin{array}{cccc} a & 0 \\
0 & a\end{array}\right) \colon a \in  \mathbb{F}_{q} \right\}.\]
\end{ex}

We now prove that the partially scattered polynomials are not far from being low weight polynomials.

\begin{prop}\label{p:weightps}
Let $t$ be a nontrivial divisor of $n$.
\begin{enumerate}[$(i)$]
\item If $f $ is a $\Rt$partially scattered polynomial in $\mathcal{L}_{n,q}$,
then $w_{L_f}(P)\le n/2$ for any point $P\in \PG(1,q^n)$.
\item If $f $ is a $\Lt$partially scattered polynomial in $\mathcal{L}_{n,q}$,
then $w_{L_f}(P)\le t$ for any point $P\in \PG(1,q^n)$.
\end{enumerate}
\end{prop}
\begin{proof}
    $(i)$ Let $P=\la v\ra_{\Fqn}$ be a point in $\PG(1,q^n)$.
    Then $P$ can be seen as an $n$-dimensional vector space over $\Fq$.
    The collection $\cD_P=\{\la w\ra_{\Fqt}\colon w\in\la v\ra_{\Fqn},\,w\neq0\}$
    is a Desarguesian spread of $P$.
    Let $U_{P,f}=\la v\ra_{\Fqn}\cap U_f$ whose $\Fq$-dimension is the weight of $P$.
    If $(y,f(y)),(z,f(z))\in\la w \ra_{\Fqt}\in\cD_P$, $yz\neq0$,
    then $f(y)/y=f(z)/z$ and $y/z\in\Fqt$ that by assumption implies
    $\la y\ra_{\Fq}=\la z\ra_{\Fq}$, and so 
    $(y,f(y))$ and $(z,f(z))$ belong to the same $\Fq$-subspace  of $P$.
    As a consequence, $U_{P,f}$ is an $\Fq$-subspace that is scattered with respect to
    the Desarguesian spread $\cD_P$, giving rise to a scattered $\Fq$-linear set of $\PG(\frac nt-1,q^t)$.
    By Theorem \ref{th:boundscatt}, $\dim_{\Fq}(U_{P,f})\le n/2$.
    $(ii)$ Let $y\in\F_{q^n}^*$. By definition, $U_f\cap\la(y,f(y))\ra_{\Fqn}\subseteq\la(y,f(y)\ra_{\Fqt}$,
    hence $w_{L_f}(\la(y,f(y))\ra_{\Fqn})\le t$.
\end{proof}

\begin{comment}
%Two polynomials $f,g\in\cL_{n,q}$ are said to be \emph{equivalent}
%under the action of $\GL(2,q^n)$ if $U_f$ and $U_g$
%belong to the same orbit under the action of $\GL(2,q^n)$.
Clearly, if $f$ and $g$ are equivalent, then $\mathbb S_f\cong\mathbb S_g$.
\textcolor{red}{In particular, two $q$-polynomials $f(x)$ and $g(x)$ such that $g(x)=af(bx)+cx$ for some
$a,b,c\in \mathbb{F}_{q^n}$, $a,b\neq0$, are projectively equivalent.}\comment[id=FERD]{se non ho capito male, non serve questa definizione ristretta...}
\end{comment}

\subsection{Non-low weight $\Lt$partially scattered polynomials}

The next two results are stated in order to characterize the $\Lt$partially
scattered polynomials $f $ whose related algebras $\mathbb S_f$ are not fields.

\begin{prop}
  Let $t$ be a positive integer and let $\ell \in\cL_{t,q}$.
  Furthermore let $\tau\in\F_{q^{2t}}$ such that $\tau^{q^t}+\tau=0$ with $\tau\neq0$.
  Then the polynomial $f =\ell ^{q^t}-\ell $ is $\Lt$partially scattered if and only
  if the following polynomial in $\cL_{t,q}$ is nonsingular:
  \begin{equation}\label{e:pinv}
  \begin{cases} \ell &\text{ for }q\text{ even,}\\ \tau\ell(\tau x)&\text{ for }q\text{ odd.}
  \end{cases}
  \end{equation}
\end{prop}
\begin{proof}
We will repeteadly use the fact that if $\Tr_{q^{2t}/q^t}(a)=\Tr_{q^{2t}/q^t}(b)=0$, then $(i)$~$ab\in\Fqt$, $(ii)$~$\Tr_{q^{2t}/q^t}(a^q)=0$.

For any $x\in\mathbb F_{q^{2t}}$,
\[
(\tau^{-1}f )^{q^t}=
\tau^{-q^t}(\ell -\ell ^{q^t})=(-\tau^{-1})(-f )=\tau^{-1}f .
\]
As a consequence, the image of $f $ is contained in $\tau\Fqt$.

Assume $\frac{f(y)}y=\frac{f(z)}z$, $y,z\in\F_{q^{2t}}^*$;
if $f(z)\neq0$ then $\frac yz=\frac{f(y)}{f(z)}\in\Fqt$.
Therefore, since $\Fqt\subseteq\ker (f )$, $f $ is $\Lt$partially scattered if and only if
$\dim_{\fq}(\ker (f ))\le t$.
The polynomial $\ell $ is of the form
\[
  \ell =\sum_{i=0}^{t-1}c_ix^{q^i},\quad c_i\in\Fqt,\ i=0,1,\ldots,t-1.
\]

Next suppose that \underline{$q$ is odd}, then for any $x\in\F_{q^{2t}}$ there are unique
$x_1,x_2\in\Fqt$ such that $x=x_1+x_2\tau$.
Since
\begin{equation}\label{e:tauFq1}
f(x)=f(x_1)+f(x_2\tau)=f(x_2\tau)=\sum_{i=0}^{t-1}c_i(x_2\tau)^{q^{i+t}}
-\sum_{i=0}^{t-1}c_i(x_2\tau)^{q^{i}}=-2\ell(x_2\tau),
\end{equation}
the $q$-polynomial $f $ is $\Lt$partially scattered if and only if the 
restriction of $\ell(x\tau)$ to $\Fqt$ (i.e. seen as an $\fq$-linear map from $\Fqt$ to $\fqn$) has rank $t$.

Finally let consider the case in which \underline{$q$ is even}. If $\theta\in\F_{q^{2t}}$ satisfies $\theta^{q^t}+\theta\neq0$,
then $\F_{q^{2t}}=\Fqt\oplus\theta{\Fqt}$, hence 
for any $x\in\F_{q^{2t}}$ there are unique
$x_1,x_2\in\Fqt$ such that $x=x_1+x_2\theta$.
Since
\begin{equation}\label{e:tauFq2}
f(x)=f(x_2\theta)=\sum_{i=0}^{t-1}c_i\theta^{q^{t+i}}x_2^{q^i}+
\sum_{i=0}^{t-1}c_i\theta^{q^{i}}x_2^{q^i}=\ell(x_2(\theta^{q^t}+\theta))
\end{equation}
and $\theta^{q^t}+\theta\in\Fqt$,
the $q$-polynomial $f $ is $\Lt$partially scattered if and only if the rank of 
$\ell $ as a polynomial in $\cL_{t,q}$ is $t$.
\end{proof}

%In the next result we give a complete characterization of $\Lt$partially scattered polynomials with the property that $\mathbb{S}_f$ is a field.

\begin{sloppy}
\begin{thm}
    Let $t$ be a proper divisor of $n$.
    Let $f \in\mathbb{F}_{q^{n}}[x]$ be an $\Lt$partially scattered polynomial
   in $\cL_{n,q}$. 
   Then $\mathbb{S}_f$ is not a field if and only if
   $f $ is equivalent to $\ell ^{q^{t}}-\ell $
   for some $\ell \in\cL_{t,q}$, and $n=2t$.
\end{thm}
\end{sloppy}

\begin{proof}
By Proposition~\ref{p:weightps} and Theorem~\ref{thm:Gfofield},
only the case $n=2t$ has to be taken into account.
Let $\tau\in\F_{q^{2t}}$ such that $\tau^{q^t}+\tau=0$, $\tau\neq0$.
The sufficiency follows by checking that the non-invertible matrix 
$\left(  \begin{array}{cc} 0 & \tau \\ 0 & 0 \end{array}\right)$ stabilizes 
$\mathcal{G}_{\ell^{q^{t}}-\ell}$. 
Indeed, as has been shown in \eqref{e:tauFq1}, \eqref{e:tauFq2}
in the respective cases of $q$ odd or even, for any $x\in\mathbb F_{q^{2t}}$,
$f(x) =\ell(x)^{q^t}-\ell(x) $ belongs to $\tau\Fqt$, hence
$f(\tau f(x))=0$.

Next assume that $\mathbb{S}_f$ is not a field.
Then $C\in\mathbb{F}_{q^{n}}^{2\times2}$ of rank one exists such that $C\mathcal G_f\subseteq \mathcal G_f$. 
Let $C^{1}$ and $C^{2}$ be the columns of $C$, then  for any $y\in\mathbb{F}_{q^{n}}$
there is $z\in\mathbb{F}_{q^{n}}$ such that 
$yC^{1}+f(y)C^{2}=\left(  \begin{array}{c} z \\ f(z) \end{array}\right)$. 
The condition $C^{2}=0$ would imply $\langle C^{1}\rangle_{\F_{q^{n}}}\subseteq \mathcal G_f$, contradicting
the assumption that $f $ is $\Lt$partially scattered.
So, since $C$ has rank one, then there exists $\alpha\in\Fqn$ such that $C^{1}=\alpha C^{2}$, and for all 
$x\in\mathbb{F}_{q^{n}}$ we have that 
\begin{equation}\label{e:gx}
(\alpha x+f(x))C^{2}\subseteq \mathcal{G}_f.
\end{equation}
Let $g =\alpha x+f $. 
There is at least one element $x_0\in\Fqn$ such that $g(x_{0})\neq 0$, otherwise $\frac{f(x)}{x}=-\alpha$
for any $x\in\Fqn$, once again contradicting
the assumption that $f $ is $\Lt$partially scattered.
Let $\beta=g(x_{0})^{-1}$. 
Note that $\beta g(x)\in\mathbb{F}_{q^{t}}$ for any $x\in\Fqn$ since 
$g(x)C^{2},g(x_{0})C^{2}\in \mathcal G_f$ and $f$ is $\Lt$partially scattered. 
It holds $0\neq x\in \ker g$ if and only if $\frac{f(x)}{x}=-\alpha$.
Assume that a $\omega\in\ker g$, $\omega\neq0$ exists.
Then for any $z\in\ker g$, $z\neq0$, $f(\omega)/\omega=f(z)/z$,
implying, by the assumption that $f $ is $\Lt$partially scattered,
$\ker g\subseteq \omega\Fqt$.
In particular, $\dim_{\fq}\ker g\leq t$, and the equality holds if and only if $\ker g= \omega\Fqt$.
Now let \[C^2=\begin{pmatrix}c_{12}\\ c_{22}\end{pmatrix}.
\]
By \eqref{e:gx}, $c_{12}\neq0$.
Assume that $g(y)g(z)\neq0$ for $y,z\in\fqn$.
Since $g(y)C^2$ and $g(z)C^2$ belong to $\mathcal G_f$, defining
$Y=c_{12}g(y)$ and $Z=c_{12}g(z)$ one obtains
\[
\frac{f(Y)}Y=\frac{f(Z)}Z=\frac{c_{22}}{c_{12}}
\]
and this implies $g(y)/g(z)\in\Fqt$.
Therefore, $\dim_\fq\im g\le t$.
Combining with $\dim_\fq\ker g\le t$, one obtains
$\dim_\fq \im g=\dim_\fq\ker g=t$.
This in turn implies $\ker g=\omega\mathbb{F}_{q^{t}}$ for some
$\omega\in\mathbb F_{q^n}^*$.
Define $h=\tau\beta g(\omega x)$ for $x\in\fqn$.
We have
\[
\begin{pmatrix}\omega^{-1}&0\\ \alpha\tau\beta&\tau\beta\end{pmatrix}
\begin{pmatrix}x\\ f(x)\end{pmatrix}=
\begin{pmatrix}\omega^{-1}x\\ \tau\beta g(x)\end{pmatrix}=
\begin{pmatrix}
y\\ \tau\beta g(\omega y)    
\end{pmatrix}
\]
for $y=\omega^{-1}x$; so, $f $ and $h $ are equivalent polynomials.
%It holds $\ker h\mathbb{F}_{q^{t}}$, $\im h=\tau\Fqt$.
Let $h =\sum_{i=0}^{2t-1}c_{i}x^{q^{i}}$.
Since $\ker h=\mathbb{F}_{q^{t}}$, for all  $x\in\mathbb{F}_{q^{t}}$ it holds
\[ \sum_{i=0}^{t-1}(c_{i}+c_{i+t})x^{q^{i}}=0, \] 
that implies $c_{i}+c_{i+t}=0$, $i=0,\ldots,t-1$. 
Therefore 
\[  h =\sum_{i=0}^{t-1}c_i(x^{q^i}-x^{q^{i+t}}). \]
Taking into account that $\beta g(\omega x)\in\Fqt$ for any $x\in\fqn$,
it follows $h(x)^{q^t}+h(x)=0$, and
\[
-\sum_{i=0}^{t-1}c_i^{q^t}(x^{q^i}-x^{q^{i+t}})+\sum_{i=0}^{t-1}c_i(x^{q^i}-x^{q^{i+t}})=0
\]
hence for any $i=0,1,\ldots,t-1$, $c_i$ is in $\Fqt$, leading to $h =\ell ^{q^t}-\ell $ for $\ell =-\sum_{i=0}^{t-1}c_{i}x^{q^{i}}$. 
\end{proof}

We did not manage to find a similar characterization for the
$\Rt$partially scattered polynomials such that $\mathbb S_f$ is not a field.
We are going to give an example to show that this situation can occur (Corollary~\ref{c:exRt}).

\subsection{The case of complementary weights}\label{s:cw}

Following \cite{Complwei}, we define a polynomial to have \textit{complementary weights} if $L_f$ has two points of complementary weights, i.e.\ there exist two point $P,Q \in \mathrm{PG}(1,q^n)$ with $P\ne Q$ such that 
\[ w_{L_f}(P)+w_{L_f}(Q)=n. \]

The linear sets with complementary weights can be described by projecting maps. Recall that if $S$ and $T$ are $\fq$-subspaces of $\fqn$ such that $\fqn=S \oplus T$ denote by $p_{T,S}$ the $\fq$-linear map of $\fqn$ projecting an element of $\fqn$ from $T$ onto $S$.

\begin{thm}\label{clas}\cite[Theorem 3.4]{Complwei}
Let $L_U$ be an $\fq$-linear set of rank $n$ in $\PG(1,q^n)$ admitting a point $P$ of weight $t$ and a point $Q\ne P$ of weight $s$ with $t+s=n$.
Then $L_U$ is  $\mathrm{PGL}(2,q^n)$-equivalent to the $\fq$-linear set
\[L_{p_{T,S}}=\{\langle (x,p_{T,S}(x)) \rangle_{\fqn} \colon x \in \fqn^*\},\]
where $T$ and $S$ are $\fq$-subspaces of $\fqn$ of dimension $t$ and $s$, respectively, such that $\fqn=T\oplus S$ and $p_{T,S} $ is the projection map of $\fqn$ from $T$ onto $S$.
Moreover, if $T=\langle \xi_0,\ldots,\xi_{t-1}\rangle_{\fq}$, $S=\langle \xi_t,\ldots,\xi_{n-1} \rangle_{\fq}$ and $\mathcal{B}^*=(\xi_0^*,\ldots,\xi_{n-1}^*)$ is the dual basis of $\mathcal{B}=(\xi_0,\ldots,\xi_{n-1})$, then
\[ p_{T,S} =\sum_{i=t}^{n-1}\xi_i\mathrm{Tr}_{q^n/q}(\xi_i^*x)=\sum_{j=0}^{n-1}\left( \sum_{i=t}^{n-1} \xi_i\cdot \xi_i^{*q^j} \right)x^{q^j}. \]
\end{thm}

We will now show that when $f =p_{T,S} $, for some $\fq$-subspaces $S$ and $T$ of $\fqn$, $\mathbb{S}_f$ is not a field. Note that, by definition, projection maps are not low weight polynomials.

\begin{prop}\label{prop:nofields}
Let $S$ and $T$ be two $\fq$-subspaces of $\fqn$ such that $\fqn=S \oplus T$ and both of dimension greater than zero. Then $\mathbb{S}_{p_{T,S}}$ is not a field.
\end{prop}
\begin{proof}
Let us start by observing that 
\begin{equation}\label{eq:cond1pTS} p_{S,T} =x-p_{T,S}  \end{equation}
and
\begin{equation}\label{eq:cond2pTS} p_{T,S}\circ p_{S,T} =0, \end{equation}
so that the matrix
\[ \begin{pmatrix}
    1 & -1 \\ 0 &0
\end{pmatrix} \in \mathbb{S}_{p_{T,S}}, \]
since 
\[  
\begin{pmatrix}
    1 & -1 \\ 0 &0
\end{pmatrix} \begin{pmatrix}
   x \\ p_{T,S}(x)
\end{pmatrix}=\begin{pmatrix}
   x - p_{T,S}(x) \\ 0
\end{pmatrix} \in \mathcal{G}_{p_{T,S}},
\]
by \eqref{eq:cond1pTS} and \eqref{eq:cond2pTS}.
So, $\begin{pmatrix}
    1 & -1 \\ 0 &0
\end{pmatrix}$ is a non-invertible matrix in $\mathbb{S}_{p_{T,S}}$.
\end{proof}

We will now show examples of R-$q^{n/2}$-partially scattered polynomials whose linear set has exactly two points of weight $n/2$.
To this aim we start with the following definition.

\begin{defn}
Let $U$ be an $\fq$-subspace of $\fqn\times \fqn$ and let $t$ be a divisor of $n$. Then we say that $U$ is \textit{R-$q^t$-partially scattered} if and only if for every $u \in \fqn\times \fqn$ we have
\[ \dim_{\fq}(U \cap \la u \ra_{\F_{q^t}})\leq 1. \]
\end{defn}

In particular, $f$ is R-$q^t$-partially scattered if and only if $U_f$ is R-$q^t$-partially scattered, and so our definition can be seen as a slight extension to the notion of R-$q^t$-partially scattered given for linearized polynomials in \cite{LZ2}.

We show that the following example of linear set with two points of complementary weights in \cite[Corollary 4.9]{Complwei} defines also an R-$q^t$-partially scattered subspace.

%\begin{cor}\label{cor:ex2}\cite[Corollary 4.9]{Complwei}
%Let $s$ be a positive integer coprime with $t$, $f(x)=x^{q^s}$ and $g(x)=\mu x^{q^s}$ with $\mu \in \F_{q^t}^*$ such that $\N_{q^t/q}(\mu) \neq 1$ and $\N_{q^{t}/q}(-\xi^{q^t+1}\mu)\neq (-1)^t$, 
%with $\xi \in \F_{q^{2t}}\setminus\F_{q^t}$.
%Let
%\[ U=T \times S=\{ (v+\xi \mu v^{q^s},u+\xi u^{q^s}) \colon u,v \in \F_{q^t} \}. \]
%Then $L_U$ has exactly two points of weight greater than one.
%\end{cor}

\begin{prop}\label{prop:2pointsrqtpartscatt}
Let $s$ be a positive integer coprime with $t$, $f =x^{q^s}$ and $g =\mu x^{q^s}$ with $\mu \in \F_{q^t}^*$ such that $\N_{q^t/q}(\mu) \neq 1$ and $\N_{q^{t}/q}(-\xi^{q^t+1}\mu)\neq (-1)^t$, 
with $\xi \in \F_{q^{2t}}\setminus\F_{q^t}$.
Let
\[ U=T \times S=\{ (v+\xi \mu v^{q^s},u+\xi u^{q^s}) \colon u,v \in \F_{q^t} \}. \]
Then $U$ is an R-$q^t$-partially scattered.
\end{prop}
\begin{proof}
By \cite[Corollary 4.9]{Complwei}, for every point $u \in \fqn\times\fqn \setminus\{ \la (1,0) \ra_{\fqn}\cup \la (0,1) \ra_{\fqn} \}$ we have that
\[ \dim_{\fq}(U\cap \la u \ra_{\fqn})\leq 1, \]
and clearly this property holds when replacing $\la u \ra_{\fqn}$ by $\la u \ra_{\F_{q^t}}$.
Now, assume that $u=(v,0)$ for some $v \in \F_{q^n}^*$. By assumptions, we have that $v=v_0+\xi v_1$ for some $v_0,v_1 \in \F_{q^t}$ and so
\[ \la (v,0)\ra_{\F_{q^t}}\cap U_f = \la (v,0)\ra_{\F_{q^t}}\cap T \times \{0\}, \]
and hence, since $\mu x^{q^s}$ is a scattered polynomial,
\[ \dim_{\fq}(\la (v,0)\ra_{\F_{q^t}}\cap U_f)=\dim_{\fq}(\la (v_0,v_1)\ra_{\F_{q^t}}\cap \overline{U}_{\mu x^{q^s}})\leq 1, \]
where $\overline{U}_{\mu x^{q^s}}=\{ (y,\mu y^{q^s}) \colon y \in \F_{q^t} \}$.
The same argument can be applied to the case $u=(0,v)$, obtaining the assertion.
\end{proof}

Since the equivalence preserves the $\Rt$partially scattered property, the above example give examples of R-$q^t$-partially scattered polynomials by the following result.

\begin{thm}\cite[Theorem 4.20]{Complwei}\label{thm:polscattscatt}
Let $U$ be as in the previous result.
Then $L_U$ is $\mathrm{PG L}(2,q^n)$-equivalent to $L_p$, where
\[p =\sum_{k=0}^{n-1}\left( \sum_{\ell=0}^{t-1} (u_\ell+u_\ell^{q^s}\xi ){\lambda_{\ell}^*}^{q^k} \right)x^{q^k},\]
$\{u_0,\ldots,u_{t-1}\}$ is an $\fq$-basis of $\F_{q^t}$ and $(\lambda_0^*,\ldots,\lambda_{n-1}^*)$ is the dual basis of $(u_0+\mu u_0^{q^s}\xi,\ldots,u_{t-1}+ \mu u_{t-1}^{q^s}\xi,u_0+u_0^{q^s}\xi,\ldots,u_{t-1}+u_{t-1}^{q^s}\xi)$.
\end{thm}

Since the property of being R-$q^t$-partially scattered is invariant by equivalence, then Theorem \ref{thm:polscattscatt} gives an examples of R-$q^t$-partially scattered whose associated linear set has exactly two points of weight $n/2$, because of Proposition \ref{prop:2pointsrqtpartscatt}.
Therefore, plugging together the above results, we have the following.

\begin{cor}\label{c:exRt}
There exist examples of R-$q^t$-partially scattered polynomials $f$ with the property that $(\mathbb{S}_f,+,\cdot)$ is not a field.
\end{cor}
\begin{proof}
The polynomial as in Theorem \ref{thm:polscattscatt} is R-$q^t$-partially scattered polynomial by Proposition \ref{prop:2pointsrqtpartscatt}, for which the associated $(\mathbb{S}_f,+,\cdot)$ is not a field by Proposition \ref{prop:nofields}.
\end{proof}

\begin{rem}
    This example is the first example of $\Rt$partially scattered polynomial that satisfies the equality in the lower bound (b) of \cite[Proposition 3.1]{LZ2}.
\end{rem}

\section{Right idealizer of two-dimensional linear rank-metric codes}\label{s:RCf}

When $m=n$, $\fq$-linear rank-metric codes can be studied in terms of linearized polynomials.
Indeed, an $\fq$-linear rank-metric code $\mathcal{C}$ is an $\F_{q}$-subspace of $ {\mathcal{L}}_{n,q}$ endowed with the rank-metric and all the notions already given for rank-metric codes can be read in this context.
Indeed, two $\fq$-linear rank-metric codes $\C_1, \C_2 \subseteq {\mathcal{L}}_{n,q}$ are \emph{equivalent} if there exist two invertible linearized polynomials $g , h  \in {\mathcal{L}}_{n,q}$ and a field automorphism $\rho \in \mathrm{Aut}(\F_{q^n})$ such that
\[
\C_1=g  \circ \C_2^{\rho} \circ h  =\{g  \circ {f^{\rho} } \circ h  : f  \in \C_2\},
\]
where $f^{\rho} =\sum_{i=0}^{n-1}\rho(f_i)x^{q^i}$ if  $f =\sum_{i=0}^{n-1}f_i x^{q^i}$.
In addition, for an $\fq$-linear rank-metric code $\C$ in ${\mathcal{L}}_{n,q}$, the set 
\[ L(\C)=\{ g  \in  {\mathcal{L}}_{n,q} : g  \circ f  \in \C, \,\,\, \text{for any}\,\,\,   f  \in \C \} \]
corresponds to the \emph{left idealizer} of $\C$ and
\[ R(\C)=\{ g  \in  {\mathcal{L}}_{n,q} : f  \circ g   \in \C, \,\,\, \text{for any}\,\,\,   f  \in \C \} \]
corresponds to the \emph{right idealizer} of $\C$.
We say that a linear rank-metric code $\C$ is an $\F_{q^n}$-\emph{linear rank-metric code} if $(L(\C),+,\circ)$, where $+$ and $\circ$ are the sum and the composition of linearized polynomials respectively, contains a subring isomorphic to 
\[\mathcal{F}_n=\{\alpha x : \alpha \in \F_{q^n}\},\]
which is an $\fq$-algebra isomorphic to $\fqn$.
As shown in \cite[Remark III.3]{PSSZ23}, if $\C$ is an $\F_{q^n}$-linear rank-metric code in ${\mathcal{L}}_{n,q}$ then it is equivalent to an $\F_{q^n}$-linear rank-metric code $\C'$ that is also an $\F_{q^n}$-subspace of ${\mathcal{L}}_{n,q}$.
For any linearized polynomial $f \in {\mathcal{L}}_{n,q}$, we can consider the following $\fqn$-linear rank-metric code
\[ \mathcal{C}_f=\langle x, f  \rangle_{\fqn}.  \]
Sheekey in \cite{Sheekey} pointed out that $\C_f$ is an MRD code if and only if $f$ is a scattered polynomial. 
When $\C_f$ is an MRD code, we already know that its right idealizer is a field, cf.\ Theorem \ref{th:propertiesideal}. In the next result we will see that we can relax the condition of being MRD codes when considering two-dimensional $\fqn$-linear rank-metric codes.
To this aim we first prove a relation between $\mathbb{S}_f$ and $R(\C_f)$, extending \cite[Lemma 4.1]{LMTZ} where the result was proved under the assumption that $f$ is scattered.

\begin{thm}\label{thm:isoGfRCf}
Let $f \in {\mathcal{L}}_{n,q}$ and denote by $\C_f$ the associated rank-metric code in ${\mathcal{L}}_{n,q}$. 
Suppose that $f \notin \la x\ra_{\fqn}$.
Then the $\fq$-algebras $\mathbb{S}_f$ and $R(\C_f)$ are isomorphic.
\end{thm}
\begin{proof}
We claim that the isomorphism is
\[\psi \colon \begin{pmatrix}
    a & b\\
    c & d
\end{pmatrix} \mapsto ax+bf  .\]
Let $A=\begin{pmatrix}
    a & b\\
    c & d
\end{pmatrix} \in \mathbb{S}_f$, and define
$g_A =ax+bf $.
The condition $A\in\mathbb S_f$ implies $f\circ g_A =cx+df \in\C_f$.
In combination with $x\circ g_A \in\C_f$, this implies
$g_A \in R(\C_f)$.
So, $\psi$ maps $\mathbb S_f$ into $R(\C_f)$.
Conversely, for any $g \in R(\C_f)$ there are unique $a,b,c,d \in \fqn$ such that
\[ g =ax+bf \,\,\,\text{and}\,\,\, f\circ g =cx+df . \]
Therefore
\[
\begin{pmatrix}
    a & b\\
    c & d
\end{pmatrix}
\begin{pmatrix}
    x\\
    f(x)
\end{pmatrix}
= \begin{pmatrix}
    y\\
    f(y)
\end{pmatrix}\quad\text{for }y=g(x).
\]
This implies that $\psi$ is bijective.

As regards the fact that $\psi$ preserves the algebra operations, the only nontrivial check remaining is $g_{AA'} =g_A \circ g_{A'} $ for any $A,A'\in\mathbb S_f$.
Let $A$ be as above, and $A'=\begin{pmatrix}
    a' & b'\\
    c' & d'
\end{pmatrix}.$
Then
\begin{align*}
g_{AA'} &=(aa'+bc')x+(ab'+bd')f =\\
&=ag_{A'} +b(c'x+d'f )=ag_{A'} +bf\circ g_{A'}=\\
&=g_{A}\circ g_{A'} 
\end{align*}
as desired.
\end{proof}

As a consequence we obtain the following.

\begin{cor} Let $f$ be a linearized polynomial in $\mathcal L_{n,q}$.
If $d(\C_f)>n/2$, then $R(\C_f)$ is a field. 
\end{cor}
\begin{proof}
    By \cite[Theorem 2]{Randr} it follows that $w_{L_f}(P)<n/2$ for any point $P$, that is $f$ has low weight, and the assertion then follows from Theorems \ref{thm:Gfofield} and \ref{thm:isoGfRCf}.
\end{proof}

\begin{rem}
    The above corollary is part of \cite[Lemma 4.4]{LTZ2}.
\end{rem}

%\begin{ex}
%Let $n=tt'$ and let $f(x)=\mathrm{Tr}_{q^n/q^t}(x)=x+x^{q^t}+\ldots+x^{q^{n-t}}$.
%First observe that $d(\C_{\mathrm{Tr}_{q^n/q^t}})=t\leq n/2$ and let's compute its right idealizer.
%Let $g(x) \in \C_{\mathrm{Tr}_{q^n/q^t}}$. Since $x,\mathrm{Tr}_{q^n/q^t}(x) \in \C_{\mathrm{Tr}_{q^n/q^t}}$ there exist $a,b,c,d \in \fqn$ such that
%\[ g(x)=ax+b\mathrm{Tr}_{q^n/q^t}(x)\,\,\,\text{and}\,\,\,\mathrm{Tr}_{q^n/q^t}(f(x))=cx+d\mathrm{Tr}_{q^n/q^t}(x). \]
%From which we get
%\[ cx+d\mathrm{Tr}_{q^n/q^t}(x)=\mathrm{Tr}_{q^n/q^t}(ax+b\mathrm{Tr}_{q^n/q^t}(x)) \]
%\[= \mathrm{Tr}_{q^n/q^t}(ax)+\mathrm{Tr}_{q^n/q^t}(b)\mathrm{Tr}_{q^n/q^t}(x).\]
%By comparing the coefficients of $x$, $x^{q^t}$ and $x^{q^{2t}}$ we obtain
%\[ \left\{ \begin{array}{llll}
%a+\mathrm{Tr}_{q^n/q^t}(b)=c,\\
%a^{q^t}+\mathrm{Tr}_{q^n/q^t}(b)=d,\\
%a^{q^{2t}}+\mathrm{Tr}_{q^n/q^t}(b)=d,\\
%\end{array}
%\right. \]
%hence $a \in \F_{q^t}$, $d=a+\mathrm{Tr}_{q^n/q^t}(b)$ and $c=d$.
%Therefore,
%\[ R(\C_{\mathrm{Tr}_{q^n/q^t}})=\{ a\mathrm{Tr}_{q^n/q^t}(x) \colon a \in \F_{q^t} \} \]
%and it is not a field as all of its elements are non-invertible.
%\end{ex}

\begin{ex}
Let $n=tt'$ with $t'>2$, and let $f =\mathrm{Tr}_{q^n/q^t}(x)=x+x^{q^t}+\ldots+x^{q^{n-t}}$. Thanks to the isomorphism $\psi$ provided in the proof of Theorem \ref{thm:isoGfRCf} and Example \ref{ex:traceSf}, the right idealizer of $\mathcal{C}_f$ is
\[R(\mathcal{C}_f)=\{ax+bf \colon a \in \mathbb{F}_{q^{n_{t'}}}, b \in \fqn\},\]
where $n_{t'}=t$ if $t'>2$ and $n_{t'}=n$ if $t'=2$.
Such a $R(\mathcal{C}_f)$ is not a field because $\mathbb{S}_f$ is not.
\end{ex}

%\textcolor{red}{
%The previous example suggests the following result.
%\begin{thm}
%Let $f(x)\in\Fqn[x]$ be a $q$-polynomial not of type $ax$, $a\in\Fqn$, and
%let $\cA_f$ be the set af all $A\in\Fqn^{2\times2}$ such that
%$AU_f\subseteq U_f$.
%Then $\cA_f$ and $R(\C_{f})$ are isomorphic $\Fq$-algebras.
%\end{thm}
%}

\subsection{The right idealizer of the MRD codes associated with partially scattered polynomials}

In Theorem 3.3 of \cite{LZ2}, the authors showed that if $n=tt'$ and $f \in \mathcal{L}_{n,q}$ is an R-$q^t$-partially scattered polynomial then
\[ \tilde{\C}_f=\{ F_{|\mathbb{F}_{q^t}}\colon \F_{q^t}\rightarrow \fqn \colon F \in \C_f \} \]
is an MRD code with parameters $(n,t,q;t-1)$ The left idealizer $L(\tilde{\C}_f)$ contains a copy of $\fqn$ and, since $t\leq n$, $R(\tilde{\C}_f)$ is a field with $|R(\tilde{\C}_f)|\leq q^t$ by Theorem \ref{th:propertiesideal}.

In the next result we find a relation between the right idealizer of $\C_f$ and those of $\tilde{\C}_f$.
In order to do so, we need some preliminary results.

\begin{lem}
  Let $f  \in \mathcal{L}_{n,q}$ with $f \notin \langle x \rangle_{\fqn}$ and such that $f$ is $\Rt$partially scattered. Then every $h \in \mathcal{C}_f\setminus \langle x \rangle_{\fqn}$ is an $\Rt$partially scattered polynomial.
\end{lem}
\begin{proof}
    Suppose that $h =ax+bf $, for some $a,b \in \fqn$ with $b\ne 0$. Suppose that for some $y,z \in \mathbb{F}_{q^n}^{*}$ with $y/z \in \mathbb{F}_{q^t}$ we have that $h(y)/y=h(z)/z$, then
    \[ \frac{ay+bf(y)}y=\frac{az+bf(z)}z \]
    implies $f(y)/y=f(z)/z$ and, since $f $ is $\Rt$partially scattered we have that $y/z \in \mathbb{F}_{q}$.
\end{proof}

\begin{lem}\label{lem:2}
    Let $f  \in \mathcal{L}_{n,q}$ with $f \notin \langle x \rangle_{\fqn}$ and such that $f$ is $\Rt$partially scattered. Then $h=h'$ for any $h ,h'  \in \C_f$ such that $h_{|\mathbb{F}_{q^t}}=h'_{|\mathbb{F}_{q^t}}$ (that is, they coincide when evaluating over $\F_{q^t}$).
\end{lem}
\begin{proof}
    Let $g =h -h' =ax+bf $, for some $a,b \in \fqn$. Since for any $y,z \in \F_{q^t}^*$ we have that $g(y)/y=g(z)/z$ then $g $ is not $\Rt$partially scattered. So, the above lemma,we have that $b=0$ and
    \[ g =h -h' =ax. \]
    In particular, $h(1)-h'(1)=a$ and this is zero, because of the assumptions.
\end{proof}

The next lemma is easy to check.

\begin{lem}\label{lem:3}
     Let $f  \in \mathcal{L}_{n,q}$ with $f \notin \langle x \rangle_{\fqn}$ and such that $f$ is $\Rt$partially scattered. If $\ell  \in R(\tilde{\mathcal{C}}_f)$ then there exists $m  \in \mathcal{C}_f$ such that $l=m_{|\mathbb{F}_{q^t}}$.
\end{lem}

We are now ready to establish a relation between the right idealizers of the codes $\C_f$ and $\tilde{\C}_f$. To this aim we will need the following notation.

Let $\mathcal{L}_{t,n,q}=\{g\in\mathcal{L}_{n,q}\mid g(\mathbb{F}_{q^{t}})=\mathbb{F}_{q^{t}}\}$ be the $\fq$-vector space of the $\mathbb{F}_q$-endomorphisms of $\mathbb{F}_{q^{n}}$ which fix setwise $\mathbb{F}_{q^{t}}$. Then consider the equivalence relation $\approx$, such that $g\approx g'$ if and only if $g_{|\mathbb{F}_{q^{t}}}=g'_{|\mathbb{F}_{q^{t}}}$. Then the projection map \[\tilde{\pi}:\mathcal{L}_{n,q}\longrightarrow\mathcal{L}_{n,q}/\approx,\] maps $\mathcal{L}_{t,n,q}$ onto a vector space isomorphic to $\mathcal{L}_{t,q}$, and let \[\Phi:\tilde{\pi}(g) \in \mathcal{L}_{t,n,q}/\approx \longrightarrow g_{|\mathbb{F}_{q^{t}}} \in \mathcal{L}_{t,q}.\] 

\begin{thm}\label{thm:rightideas}
Let $f  \in \mathcal{L}_{n,q}$ with $f \notin \langle x \rangle_{\fqn}$ and such that $f$ is $\Rt$partially scattered. 
Consider 
\[\C_f=\langle x,f \rangle_{\fqn} \subseteq \mathcal{L}_{n,q}\]
and 
\[\tilde{\C}_f=\{ F_{|\F_{q^t}}\colon \F_{q^t}\rightarrow \fqn \colon F \in \C_f\} \subseteq \mathrm{Hom}_{\fq}(\F_{q^t},\fqn).\]
Then
\[  (\Phi\circ \tilde{\pi})(R(\C_f)\cap \mathcal{L}_{t,n,q})\subseteq R(\tilde{\C}_f). \]
\end{thm}
\begin{proof}
Let $g\in R(\C_f)\cap \mathcal{L}_{t,n,q}$, so that $\Phi(\tilde{\pi}(g))=g_{|\mathbb{F}_{q^{t}}}$. 
We show that for any $k \in \tilde{\C}_f$ we have that $k\circ g_{|\mathbb{F}_{q^t}}\in \tilde{\C}_f$.
Let $k\in\tilde{\C}_f$, and $h\in \C_f$ such that $k=h_{|\mathbb{F}_{q^{t}}}$. Hence,
\[ k\circ g_{|\mathbb{F}_{q^{t}}}=k_{| \mathbb{F}_{q^{t}}}\circ g_{|\mathbb{F}_{q^{t}}}=h_{|\mathbb{F}_{q^{t}}}\circ g_{|\mathbb{F}_{q^{t}}}=h\circ g_{|\mathbb{F}_{q^{t}}},\] 
and since $h\circ g\in \C_f$, $k\circ(g_{|\mathbb{F}_{q^{t}}})\in\tilde{\C}_f$.

%\comment[id=FERD]{Ancora non funziona...}
%On the other hand, take $\ell\in R(\tilde{\C}_f)$.
%We want to prove that there exists $m\in R(\C_f)\cap \mathcal{L}_{t,n,q}$ such that $\ell=m_{|\mathbb{F}_{q^t}}$. 
%By Lemma \ref{lem:3} there exists $m \in \C_f$ such that $\ell=m_{|\mathbb{F}_{q^t}}$.
%Clearly, $m \in \mathcal{L}_{t,m,q}$.
%Now, consider any $h \in \C_f$, then $h_{|\mathbb{F}_{q^t}} \in \tilde{\C}_f$ and also there exists $h'(x) \in \C_f$ such that
%\[ h_{|\Fqt}\circ \ell=h'_{|\Fqt}. \]
%This implies that
%\[ h'_{|\Fqt}=h\circ \ell=h\circ m_{|\Fqt}, \]
%and from Lemma \ref{lem:2} we get $h'=h \circ m$, that is $m \in R(\C_f)$.

%Then by hypothesis $\forall h\in C_f$, $\exists h'\in C_f$ such that $h_{|\mathbb{F}_{q^{t}}}\circ \ell=h'_{|\mathbb{F}_{q^{t}}}$, and it is equivalent to $(h\circ\ell)_{|\mathbb{F}_{q^{t}}}=h'_{|\mathbb{F}_{q^{t}}}$ (recall $R(\tilde{(C_f)}$ is a field).
%The right idealizer of $\tilde{\C}_f$ is given, in terms of linear maps,
%\[ R(\tilde{\C}_f)=\{ g \in \mathrm{End}_{\fq}(\F_{q^t}) \colon h \circ g \in \tilde{\C}_f,\,\,\text{for any}\,\, h \in \tilde{\C}_f \}, \]
%which can be clearly seen as a subset of $\mathcal{L}_{t,q}$.
%So, $R(\tilde{\C}_f)$ is the set of elements in $\mathcal{L}_{t,q}$ which is also in the right idealizer of $\C_f$.
\end{proof}

Note that since $f$ is $\Rt$partially scattered, the restriction of $\Phi\circ\tilde{\pi}$ to $R(\C_f)\cap\mathcal{L}_{t,n,q}$ is injective. In fact, for any $h\in \C_f\setminus \{0\}$, $\dim_{\fq}(\ker(h)\cap\mathbb{F}_{q^{t}})\leq1$ and for any $k\in \tilde{\C}_f\setminus \{0\}$, $\dim_{\fq}(\ker(k))\leq 1$. A corollary of this property is the following.

\begin{cor}
Let $f  \in \mathcal{L}_{n,q}$ with $f \notin \langle x \rangle_{\fqn}$ and such that $f$ is $\Rt$partially scattered. Then 
\begin{equation}\label{eq:ineqrightideal}
|R(\tilde{\C}_f)|\geq |\mathcal{L}_{t,n,q}\cap R(\C_f)|.
\end{equation}
%Moreover, the equality in this bound holds if and only if $R(\tilde{\C}_f)$ and $R(\C_f)$ are isomorphic fields.
%Also, if $R(\C_f)$ is not a field then the inequality in \eqref{eq:ineqrightideal} is strict.
\end{cor}

%In the case in which $f$ is an R-$q^t$-partially scattered such that $d(\C_f)>n/2$ then $R(\C_f)$ is a field with size at most $q^n$ with $\mathcal{L}_{t,q}\cap R(\C_f)$ as a subfield of order at most $q^t$.

We will now see some examples in which we compute the two cardinalities.

\begin{ex}
Let consider $f =x^{q^s} \in \mathcal{L}_{n,q}$ with $n=tt'$ and $\gcd(s,t)=1$.
It is easy to see that $R(\C_f)=\{\alpha x \colon \alpha \in \fqn\}$. It holds \ref{thm:rightideas},
\[(\Phi\circ\tilde{\pi}) (\mathcal{L}_{t,n,q}\cap R(\C_f))=\{\alpha {x}_{|\mathbb{F}_{q^t}} \colon \alpha \in \F_{q^t}\}. \]
Furthermore,
\[ R(\tilde{\C}_f)=\{ \ell  \in \mathcal{L}_{t,q} \colon h \circ \ell \in \tilde{\C}_f,\,\,\forall h \in \tilde{\C}_f \}\]
\[=\{ \ell  \in \mathcal{L}_{t,q} \colon (ax+bx^{q^s})_{|\mathbb{F}_{q^t}} \circ \ell \in \tilde{\C}_f,\,\,\forall a,b \in \fqn \}, \]
and by Lemma \ref{lem:3} $\ell $ is of type ${\alpha x +\beta x^{q^s}}_{|\Fqt}$, which implies that $\alpha,\beta \in \Fqt$, since for any $x \in \mathbb{F}_{q^t}$ we have
\[\alpha^{q^t}x+\beta^{q^t}x^{q^s}=\alpha x+\beta x^{q^s}.\]
Next, ${\alpha x +\beta x^{q^s}}_{|\Fqt} \in R(\tilde{\C}_f)$ if and only if for any $a,b \in \fqn$ there exist $c,d \in \fqn$ such that
\begin{equation}\label{eq:rightideapseudo} {a \alpha x+a \beta x^{q^s}+\alpha^{q^s}b x^{q^s}+\beta^{q^s} b x^{q^{2s}}}_{|\Fqt}={cx+dx^{q^s}}_{|\Fqt}.  \end{equation}
For $t=2$, \eqref{eq:rightideapseudo} is equivalent to ask that for every $a,b \in \fqn$ there exist $c,d \in \fqn$ such that
\[ \left\{
\begin{array}{ll}
a \alpha + b \beta^{q^s}=c,\\
a \beta + b \alpha^{q^s} =d,
\end{array}
\right.\]
which admits always a solution. Therefore, $R(\tilde{\mathcal{C}}_f)=\mathcal{L}_{2,q}$.
If $t \ne 2$ then \eqref{eq:rightideapseudo} implies that $\beta=0$ and so
\[ R(\tilde{\C}_f)=\{\alpha {x}_{|\mathbb{F}_{q^t}} \colon \alpha \in \F_{q^t}\}.\]
\end{ex}

We will now analyze the case of LP polynomials.

\begin{ex}
Let consider $f =x^{q^s}+\delta x^{q^{s(n-1)}} \in \mathcal{L}_{n,q}$ with $n=tt'$, $t>4$ and $\gcd(s,t)=1$.
Therefore,
\[ \mathcal{L}_{t,n,q}\cap R(\C_f)=\{\alpha x \colon \alpha \in \mathbb{F}_{q^{\gcd(t,2)}}\}. \]
Moreover,
\[R(\tilde{\C}_f)=\{ \ell  \in \mathcal{L}_{t,q} \colon (ax+bx^{q^s}+b\delta x^{q^{s(n-1)}})_{|\mathbb{F}_{q^t}} \circ \ell \in \tilde{\C}_f,\,\,\forall a,b \in \mathbb{F}_{q^n} \}.\]
Let $l  \in R(\tilde{\C}_f)$, then by Lemma \ref{lem:3} there exist $\alpha,\beta \in \mathbb{F}_{q^n}$ such that
\[ l=\alpha x+\beta x^{q^s}+\beta \delta x^{q^{s(n-1)}}_{|\mathbb{F}_{q^t}}.  \]
Since $l  \in \mathcal{L}_{t,q}$, we have that for any $x \in \mathbb{F}_{q^t}$
\[ \alpha^{q^t}x+\beta^{q^t}x^{q^s}+\beta^{q^t}\delta^{q^t}x^{q^{s(t-1)}}=\alpha x+\beta x^{q^s}+\beta \delta x^{q^{s(t-1)}}, \]
from which it follows that $\alpha,\beta,\beta\delta \in \mathbb{F}_{q^t}$.
If $\delta \notin \mathbb{F}_{q^t}$ then $\beta=0$ and so $l=\alpha x_{|\mathbb{F}_{q^t}}$ with $\alpha \in \mathbb{F}_{q^t}$. For any $a,b \in \mathbb{F}_{q^n}$ there exist $c,d \in \mathbb{F}_{q^n}$ such that
\[ a\alpha x+b\alpha^{q^s}x^{q^s}+b\delta \alpha^{q^{s(t-1)}}x^{q^{s(t-1)}}=cx+dx^{q^s}+d\delta x^{q^{s(t-1)}}, \]
for any $x \in \mathbb{F}_{q^t}$.
From the above identity we get
\[\left\{
\begin{array}{l}
c=a\alpha, \\
d=b\alpha^{q^s},\\
d=b\alpha^{q^{s(t-1)}},
\end{array}
\right.\]
from which we get $\alpha \in \mathbb{F}_{q^{\gcd(t,2)}}$.
Hence, 
\[R(\tilde{\C}_f)=\{ \alpha x \colon  \alpha \in \mathbb{F}_{q^{\gcd(t,2)}} \}.\]
Arguing similarly, but with more involved computations, one gets the same conclusion for the case $\delta \in \mathbb{F}_{q^t}$.
\end{ex}

In both examples, apart from the case $t=2$, we have that $R(\tilde{\C}_f)$ and $ \mathcal{L}_{t,n,q}\cap R(\C_f)$ are isomorphic fields.

\section*{Acknowledgements}
The last two authors were partially supported by the Italian National Group for Algebraic and Geometric Structures and their Applications (GNSAGA - INdAM). The research of the third author was supported by the project ``VALERE: VAnviteLli pEr la RicErca" of the University of Campania ``Luigi Vanvitelli'' and by the project COMBINE from ``VALERE: VAnviteLli pEr la RicErca" of the University of Campania ``Luigi Vanvitelli''.

Valentino Smaldore and Corrado Zanella\\
Dipartimento di Tecnica e Gestione dei Sistemi Industriali\\
Universit\`a degli Studi di Padova\\
Stradella S. Nicola, 3\\
36100 Vicenza VI - Italy\\
{\em \{valentino.smaldore,corrado.zanella\}@unipd.it}

\smallskip

Ferdinando Zullo\\
Dipartimento di Matematica e Fisica,\\
Universit\`a degli Studi della Campania ``Luigi Vanvitelli'',\\
Viale Lincoln 5,\\
81100 Caserta CE - Italy\\
{{\em ferdinando.zullo@unicampania.it}}

\end{document}